\RequirePackage{fix-cm}
\documentclass[smallextended]{svjour3}       
\smartqed  
\usepackage{graphicx}
\usepackage{amsfonts}
\usepackage{latexsym, times}
\usepackage{graphics}
\usepackage{graphicx,float,latexsym,shortvrb}
 \journalname{arXiv.org}

\begin{document}

\title{A Novel Extension of Randomly Weighted Average
}


\author{Hazhir Homei}


\institute{H. Homei \at
              Department of Statistics, Faculty of Mathematical
Sciences, \\ University of Tabriz,  P.O.Box 51666--17766, Tabriz, Iran. \\
              Tel.: +98-411-3392863         \ \        Fax: +98-411-3342102\\
              \email{homei@tabrizu.ac.ir}           
}

\date{24 August 2013}

\maketitle

\begin{abstract}
We study a well-known problem concerning a random variable $Z$
uniformly distributed between two independent random variables. A
new extension has been introduced for this problem and fairly large
classes of randomly weighted average distributions are  identified
by their generalized Stieltjes transforms. In this article we employ
the Schwartz distribution theory for finding distributions of this
extension; we also study some of their properties.
\end{abstract}

\section{Introduction}
\label{intro}
Van Assche (1987) introduced the notion of a random variable $Z$
uniformly distributed between two independent random variables $X_1$
and $X_2$, which arose in studying the distribution of products of
random $2\times2$ matrices  for stochastic search of global maxima.
By letting $X_1$ and $X_2$ to have identical distributions, he
derived that: (i) for $X_1$ and $X_2$ on $[-1,1]$,  $Z$ is uniform on
$[-1,1]$ if and only if $X_1$ and $X_2$ have arcsin distribution;
and (ii) $Z$ possesses the
 same distribution as $X_1$ and $X_2$ if
and only if $X_1$ and $X_2$ are degenerated or have a Cauchy
distribution. Soltani and Homei (2009a) extended Van Assche's results as follows: They put $X_1,\cdots,X_n$ to
be independent, and considered for $n\geq 2$
$$S_n(R_{1},...,R_{n-1}) = R_{1}X_1 + R_2X_2 + \cdots
+R_{n-1}X_{n-1} + R_nX_n,\;\;\;\;  \eqno(1.1)$$ where
random proportions are $R_{i}=U_{(i)}-U_{(i-1)}$  (for any $i\in\{1,...,n-1\}$) and
$R_n= 1-\sum_{i=1}^{n-1} R_i$, $U_{(1)},...,U_{(n-1)}$ are order
statistics from a uniform distribution on $[0,1]$, and $U_{(0)}=0$.
Soltani and Roozegar (2012), defined fairly large classes of randomly
weighted average (RWA) distributions by using new random weights.
These are cuts of $[0,1]$ by $U_{(k_1)},...,U_{(k_n-1)}$, where
$$R_{(k_{j})}=U_{(k_{j})}-U_{(k_{j-1})},\;\;\; j=1,\cdots,n, \;\;\;\textrm{and}\;\;\; U_{(k_{0})}=0\;\;\;
\textrm{and}\;\;\; U_{(k_{n})}=1,$$
and $$ S_{n^*}(k_{1},...,k_{n-1}) =\sum_{j=1}^{n}
R_{k_j}X_i,\;\;\;k_{n}=n^*_{0}.\eqno(1.2)
$$
They employ generalized Stieltjes transform for RWA's distributions
and observe some interesting results. In this article, we follow the
work of Homei (2012) for finding RWA's
distributions by using Schwartz distribution theory.

\section{Some Earlier Results}

In this section, we first review some results of Homei (2012) and
then modify them a little bit to fit in our framework for getting
Theorem~1. We provide the conditional distribution of
${S_{n^*}(k_{1},...,k_{n})}=\sum_{j=1}^n R_{k_j}X_j$ for given
$(X_1,..., X_n)$$=(x_1, ..., x_n)$ at $z$, denoted by  $k(z|x_1,
..., x_n).$ At first we assume  $x_1>x_2> ...
>x_n>0,$ but later we will remove this restriction on $x_1, ...,
x_n.$ We recall that $(U_{(1)}, ...,U_{(n)})$ is the order
statistics of a random sample $U_1, ..., U_n$ for the uniform $[0 ,
1].$ The sequence of indices $\{k_1, ..., k_{n-1}\}$ is an ordered
subsequence of $\{1, ..., n^*-1\}.$ Thus $\{U_{k_1}, ...,
U_{k_{n-1}}\} \subset \{U_{(1)}, ..., U_{(n^*-1)}\}$ and the
increments $R_{k_j}$ are defined by $R_{k_j}=U_{k_j}-U_{k_{j-1}},
j=1, ..., n-1,$ where $U_{k_{0}}=0$ and $U_{k_{n}}=1$. Since
$\sum_{j=1}^nR_{k_j}=1,$ then $k(z|x_1, ..., x_n)$ can be expressed as
$$P( \sum_{j=1}^{n-1}c_jR_{k_j} \leq z-x_n); \;\
c_j=x_j-x_n, \, j=1, ..., n-1,\eqno(2.1)$$ the distribution
$\sum_{j=1}^{n-1}c_jR_{k_j}$ was derived by Weisberg (1971) as
$$P( \sum_{j=1}^{n-1}c_jR_{k_j} \leq z-x_n
)=1-\sum_{j=1}^r \frac{h_j^{m_j-1}(x_j;z)}{(m_j-1)!},\eqno(2.2)$$
where $m_j=k_j-k_{j-1}, \ j=1, ..., n, \ k_n=n^*, \
\sum_{j=1}^nm_j=n^*$, $h_j^{(m_j-1)}(c_j)$ is the $(m_j-1)-$th
derivative of
$$h_j(x;z)=\frac{(x-z)^{n^*-1}}{c_j{\prod_{i\neq
j}^n}(x-x_i)^{m_i-1}},
\eqno(2.3)$$
at $x$ evaluated at $x_1, x_2, ..., x_n$, where $r$ is the largest positive integer with $z<x_r$.
The distribution of $\sum_{j=1}^{n-1}c_jR_{k_j}$ in (2.2) can
alternatively be expressed as
$$P(\sum_{j=1}^{n-1}c_jR_{k_j} \leq z-x_n
)=\sum_{j=r^*+1}^{n} \frac{f_{j}^{(m_j-1)}(x_j;z)}{(m_j-1)!}\;,
\eqno(2.4)$$ where $r^*$ is the largest positive integer such that
$x_{r^*} \geq z,$ and $f_j^{(m_j-1)}(x_j;z)$ are the $(m_j-1)-$th
derivatives of
$$f_j(x;z)=\frac{(x-z)^{n^*-1}}{{\prod_{i \neq
j}^n}(x-x_i)^{m_i}} \eqno(2.5)$$
at $x=x_j.$\\
By using the heaviside function, the distribution in (2.4) can be
expressed as
$$k(z\mid x_1, ..., x_n)=\sum_{j=1}^n
\frac{f_j^{(m_j-1)}(x_j;z)U(z-x_j)}{(m_j-1)!} \eqno(2.6)$$ for any
set of distinct values $x_1, ..., x_n,$ and any $$z \in [min(x_1,
..., x_n),max(x_1, ..., x_n)],$$ also
$$
k(z\mid x_1, ..., x_n)=\left\{
\begin{array}{ll}
  0 & \quad \textrm{ if } \  z<min(x_1, ..., x_n) \\
  1 & \quad \textrm{ if } \  z \geq min(x_1, ..., x_n).
\end{array}
\right.\eqno(2.7)
$$
Indeed, í$k(z;x_1, ..., x_n)$ is a big family of distributions that include
Two-Sided Power (TSP) distributions and General Two-Sided Power
 (GTSP) distributions, see Soltani and Homei (2009b).
  The conditional kernel presented in (2.6) leads
us to a fairly large class of conditional kernels. Indeed, we define
$$k(g|x_1, ...,
x_n)=\sum_{j=1}^{n}\frac{(-1)^{m_j-1}}{(m_j-1)!}\frac{d^{m_j-1}}{dx_j^{m_j-1}}\frac{g(x_j)}{\prod
_{i \neq j}(x_i-x_j)^{m_i}} \;. \nonumber\eqno(2.8)$$
For any function $g:\mathbb{R}\rightarrow \mathbb{C}$   the kernel given in (2.8) will be reduced to (2.6) if $$g_z(x)=(z-x)^{n^*-1}U(z-x).$$
 Following Van Assche (1987), we will also apply the concept of the
distribution function distributional derivative $\Lambda^{(n)}$ for
a given integer $n$, such that
$$\int_{-\infty}^{+\infty} \varphi(x)\Lambda^{(n)}(dx)=\frac{(-1)^n}{n!}\int_{-\infty}^{+\infty}
\left[\frac{d^n}{dx^n}\varphi(x)\right]\Lambda(dx). \eqno(2.9)$$ For
certain infinitely differentiable function $\varphi$, the following
lemma provides a useful integral relation between the conditional
kernel (2.8) and the distribution derivative of $F,$ the
distribution
of the random mixture ${S_{n^*}(k_{1},...,k_{n-1}) }$.

\begin{lemma}
 The $(n^*-1)$-th distributional derive of the
distribution $F$ and the conditional kernel (2.8) are subject to
$$\int_{\mathbb{R}}g(x)dF^{(n^*-1)}(x)=\int_{\mathbb{R}^{n}} k(g|x_1, ..., x_n)\prod_{i=1}^nF_{X_i}(dx_i) ,\eqno(2.10)$$
for any infinitely differential function $g$ for which the integrals
are finite.
\end{lemma}
\begin{proof}
 First we verify (2.10) with $g_z(x)=(z-x)^{n^*-1}U(z-x)$, where $z$ is a fixed real number. Indeed,
 \begin{eqnarray}
U(z-x)&=&\frac{(-1)^{n^*-1}}{(n^*-1)!}\frac{d^{n^*-1}}{dx^{n^*-1}}\left[(z-x)^{n^*-1}U(z-x)
\right]
\nonumber \\
&=&\frac{(-1)^{n^*-1}}{(n^*-1)!}\frac{d^{n^*-1}}{dx^{n^*-1}}g_z(x),\nonumber
\end{eqnarray}
thus
\begin{eqnarray}
P({S_{n^*}(k_{1},...,k_{n-1}) }\leq z)&=&\int_{\mathbb{R}} U(z-x)d F(x)\nonumber \\
&=&\int_{\mathbb{R}^n}k(z|x_1, ..., x_n) \prod_{i=1}^nF_{X_i}(dx_i),
\nonumber
\end{eqnarray}
therefore,
$$
\frac{(-1)^{n^*-1}}{(n^*-1)!}\int_{\mathbb{R}}
\left[\frac{d^{n^*-1}}{dx^{n^*-1}}g_z(x) \right]d
F(x)=\int_{\mathbb{R}^n} k(g_z|x_1, ..., x_n)
\prod_{i=1}^nF_{X_i}(dx_i).\eqno(2.11)$$

Now (2.11) together with (2.9) will lead us to

$$\int_{\mathbb{R}}g_z(x) dF^{(n^*-1)}(x)=\int_{\mathbb{R}^{n}}k(g_z|x_1, ...,
x_n) \prod_{i=1}^nF_{X_i}(dx_i) .\eqno(2.12)\hfill \Box $$
\end{proof}

Since the conditional kernel $k(g\mid x_1, ..., x_n)$ is linear in
$g$, by using an argument similar to the one given by Van Assche
(1987) we can enlarge the class of functions for which (2.10) holds for the
class of infinitely differentiable functions where the
corresponding integrals are finite. The Stieltjes transform has
appeared to be an appropriate tool for investigating how the
distributions of the random mixture ${S_{n^*}(k_{1},...,k_{n-1})}$
is related to the distributions of $X_1, ..., X_n$. The Stieltjes
transform of a distribution $H$ is defined by
$$S(H,z)=\int_{\mathbb{R}} \frac{1}{z-x}H(dx), \ z \in \mathbb{C}\cap
(suppH)^c,\eqno(2.13)$$
where $suppH$  stands for the support of $H$. The following is the main
theorem of this sections that expresses the Stieltjes transform of
the random mixture with respect to those of $X_1, ..., X_n$.

\begin{theorem}
Assume $X_1, ..., X_n$ are independent and
continuous. Then for any complex number $z\in \mathbb{C}\bigcap_{i=1}^n
(\mbox{supp}F_{X_i})^c$  the following identity holds:
$$\frac{(-1)^{n^*-1}}{(n^*-1)!}\frac{d^{n^*-1}}{dz^{n^*-1}}S(F,z)=\prod_{i=1}^n
\frac{(-1)^{m_i-1}}{(m_i-1)!} \frac{d^{m_i-1}}{dz^{m_i-1}}
S(F_{X_i},z).$$
\end{theorem}
\begin{proof}
 Let $h_z(x)=\frac{1}{z-x}.\;$ From (2.12) we have
$$S(F^{(n^*-1)},z)=\int_{\mathbb{R}^{n}} k(h_z|x_1, ..., x_n)\prod_{i=1}^n
F_{X_i}(dx_i).$$
 On the other hand
 $$k(h_z|x_1, ..., x_n)=\sum_{j=1}^n
 \frac{(-1)^{m_j-1}}{(m_j-1)!}\;\frac{d^{m_j-1}}{dx_j^{m_j-1}}\frac{1}{(z-x_j)\prod_{i \neq j}^n(x_i-x_j)^{m_i}},$$
which is equal to
$$-\prod_{j=1}^n \frac{(-1)^{m_j}}{(z-x_j)^{m_j}}\;.$$
Therefore,
$$S(F^{(n^*-1)},z)=-\prod_{j=1}^n (-1)^{m_j}
\int_{\mathbb{R}} \frac{1}{(z-x_j)^{m_j}}F_{X_j}(dx_j).$$
But since
$$\frac{d^{m_j-1}}{dz^{m_j-1}}\;\frac{1}{z-x}={(-1)^{m_j-1}}{(m_j-1)!}\;\frac{1}{(z-x)^{m_j}}\;,$$
we obtain that
$$S(F^{(n^*-1)},z)=(-1)^{n^*-1}\prod_{j=1}^n\frac{(-1)^{m_j-1}}{(m_j-1)!}
\frac{d^{m_j-1}}{dz^{m_j-1}}S(F_{X_j},z). $$
 Also we note
that
\begin{eqnarray}
S(F^{(n^*-1)},z)&=&\int \frac{1}{z-x}F^{(n^*-1)}(dx) \nonumber\\
&=&\frac{(-1)^{n^*-1}}{(n^*-1)!}\int (n^*-1)!\;\frac{1}{(z-x)^{n^*}}F(dx)\nonumber\\
&=&\frac{1}{(n^*-1)!}\frac{d^{n^*-1}}{dz^{n^*-1}}\int \frac{1}{z-x}F(dx)\nonumber\\
&=& \frac{1}{(n^*-1)!}\frac{d^{n^*-1}}{dz^{n^*-1}}S(F,z), \nonumber
\end{eqnarray}
giving the result. The proof of the theorem is now complete.\hfill$\Box$\end{proof}

\section{Some properties}
In this section some examples for applications of Theorem~1 are presented, and then some properties are introduced which follow from the properties of $X$'s.
\subsection{Characterization}
Let us now apply Homei's results in (1.2).

\begin{theorem}
Let $X_1$, $X_2$ and $X_3$ be i.i.d random
variables on $[-1,1].$ Then

\noindent (i)  For $m_1$=$m_2$=$m_3$=$1$ we have

\quad (a) $S_3$ has semicircle distribution on $[-1,1]$ if and only if
$X_1$, $X_2$ and $X_3$ have arcsin distribution.

\quad (b) $X_1$, $X_2$ and $X_3$ have semicircle distribution on
$[-1,1]$ if and only if $S_3$ has power semicircle distribution on
$[-1,1]$, i.e.,
$$f(z)=\frac{16}{5\pi}(1-z^2)^{\frac{5}{2}}\;,\;\;\;\;-1\leq z\leq
1.$$

\quad (c) $X_3$ has arcsin distribution if and only if  $S_3$ has power
semicircle distribution on $[-1,1]$, i.e.,
$$f(z)=\frac{8}{3\pi}(1-z^2)^{\frac{3}{2}}\;,\;\;\;\;-1\leq z\leq
1.$$

\noindent (ii) For $m_1=3, m_2=m_3=1$ we have

\quad (d) if $X_2, X_3$ have arcsin distribution, then $X_3$ has
semicircle distribution on $[-1,1]$ if and only if  $S_3$  has power
semicircle distribution on $[-1,1]$, i.e.,
$$f(z)=\frac{8}{3\pi}(1-z^2)^{\frac{3}{2}}\;,\;\;\;\;-1\leq z \leq
1.$$

\noindent (iii) For $m_1=1, m_2=1, m_3=2$ we have

\quad (e) $X_1$, $X_2$ and $X_3$
 have arcsin distribution on $[-1,1]$ if and only if $S_3$ has a
 semicircle distribution on $[-1,1]$.
 \end{theorem}

\begin{proof}
{\it (a)} For the ``only if"  part we assume that $S_3$
has
semicircle distribution on $[-1,1]$; then
$\frac{1}{2}{\cal S}^{''}(F_{Z},z) = (z^{2}-1)^{-\frac{3}{2}}.$
For the ``if" part assuming that $X_1$, $X_2$ and $X_3$ have arcsin
distribution it follows from Theorem~1 that
${\cal S}(F_{X},z) =(z^{2}-1)^{-\frac{1}{3}}.$

{\it (b)} By Theorem~1, we have
$$\frac{1}{2}{\cal S}^{''}(F_{Z},z) = 8(z-\sqrt{z^2-1})^3 \ \
\textrm{(for the ``if" part), and}$$

$${\cal S}(F_{X},z)  = (2(z-\sqrt{z^2-1}))^{\frac{1}{3}} \  \
\textrm{(for the ``only if" part)}.$$

{\it (c)} By Theorem~1, we have
$$\frac{1}{2}{\cal S}^{''}(F_{Z},z)  =  4( z-\sqrt{z^2-1})^2  {\cal
S}(F_{X_{3}},z)\; \ \textrm{(for the ``if" part), and}$$

$${\cal S}(F_{X_3},z)=\frac{1}{\sqrt{z^2-1}}, \;\ \textrm{(for the
``only if" part)}.$$

{\it (d)} By Theorem~1, we have for the `` if" part:
$$\frac{2}{(z^2-1)^{\frac{5}{2}}} = {\cal
S}^{''}(F_{X_{1}},z)\;\frac{1}{\sqrt{z^2-1}}\;\frac{1}{\sqrt{z^2-1}}
\ \textrm{; \ and for the ``only if" part:}$$
$${\cal S}(F_{X},z)=(\frac{z^{2}}{(z^{2}-1)^{
\frac{3}{2}}}-\frac{1}{\sqrt{z^2-1}})(\frac{1}{\sqrt{z^2-1}})(\frac{1}{\sqrt{z^2-1}}).$$

{\it (e)} By Theorem~1, we have (for the ``if" part)
$$\frac{1}{6}{\cal S}^{'''}(F,z)={\cal S}(F_{X_{1}},z) {\cal
S}(F_{X_{2}},z) {\cal S}^{'}(F_{X_{3}},z) =
  \frac{-z}{(z^2-1)^{\frac{5}{2}}},\; \ \textrm{and}
$$

$${\cal S}(F_{X_i},z)= (\frac{1}{\sqrt{z^2-1}})^{2} (\frac{-z}{(z^2-1)^{\frac{3}{2}}})                                \;\;  \;\;\; i=1,2,3 \; \ \textrm{(for the
``only if" part)}.$$
\hfill$\Box$
\end{proof}

The work of Homei (2012) has been summarized in Theorem~3 and Remark~1.

\begin{theorem}
 If $X_1$ and $X_2$ are
independent random variables with a common distribution $F_X$, then
the characterizations of ${S_{n^*}(k_{1})}$ for
$(m_{1}=1, m_{2}=1)$ and $(m_{1}=1, m_{2}=2)$ are identical.
\end{theorem}

\begin{proof}
We note that $X_1$ and $X_2$ have a common
distribution function $F_X$. By using Theorem~1 for $(m_{1}=1, m_{2}=2)$, we
have
$$-\frac{1}{2}{\cal S}^{''}(F_{{S_{n^*}(k_{1})}},z)={\cal S}(F_X,z){\cal
S}^{'}(F_X,z),$$ and so $$-{\cal
S}^{''}(F_{{S_{n^*}(k_{1})}},z)=\frac{d}{dz}{\cal
S}^{2}(F_X,z),$$ and $$-{\cal
S}^{'}(F_{{S_{n^*}(k_{1})}},z)={\cal S}^{2}(F_X,z).
\eqno(3.1)$$
We note that the Stieltjes transform tends to zero when $z$ is
 sufficiently large. In that case the constant in the
differential equation will be zero. The equation (3.1) is exactly
the equation obtained by Van Assche (1987) when $X_1$ and $X_2$ have
a common distribution; so his
results hold in our framework as well. \hfill$\Box$
\end{proof}

\begin{remark}
Let ${S_{n^*}}$ be the randomly weighted average
given in (1.2). Assume random variables $X_1$, $X_2$ are independent
and continuous, $X_i \sim F_{X_i},i=1, 2$. Then
$$B(n_1, n_2){\cal S}^{(n_{1}+n_{2}-1)} (F_Z, z)=-{\cal S}^{(n_{1}-1)} (F_{X_{1}}, z) {\cal S}^{(n_{2}-1)} (F_{X_{2}}, z)$$ holds for any $z\in \mathbb{C}\bigcap_{i=1}^2
(\mbox{supp}F_{X_i})^c.$\hfill$\triangle$
\end{remark}

\subsection{Limit properties}
 Let $X_{1}, X_{2},\cdots $ be a
sequence of independent identically random variables and $\{R_{i}\}$
be a sequence as (1.1). It is natural to ask weather
$S_n(R_{1},...,R_{n-1})$ converges in probability or not. The
present subsection is devoted to those sorts of questions.

\begin{lemma}
Let $\{X_{k}\}$ be a sequence of independent, identically
distributed random variables with $E(\mid X_{k} \mid)< \infty $ and
$E(
X_{k})=\mu$. Let $a_{nk}$ satisfy these conditions:\\
(i) $\lim_{n\longrightarrow\infty}a_{nk}=0$ for every $k$,\\
(ii) $\lim_{n\longrightarrow\infty}\sum
 _{k=1}^{\infty}a_{nk}=1$, and\\
(iii) $\sum_{k=1}^{\infty}a_{nk}\leq M$ for all $n$.\\
If $\max\mid a_{nk}\mid\longrightarrow0$ as $n\longrightarrow\infty$,
then $\sum_{k}a_{nk}X_{k}\longrightarrow \mu$ in probability.
\end{lemma}
\begin{proof}
See Pruitt (1966).\hfill$\Box$
\end{proof}

\begin{lemma}
Assume $\{R_i\}$ is given by (1.1). Then max ${R_i
\longrightarrow 0}$, w.p.1, as $n\longrightarrow \infty$.
\end{lemma}

\begin{proof}
It was shown that the maximum spacings tend to zero
when the sample is taken from uniform distribution on $[0,1]$, see
Devroye (1981). Then obviously  ${R_i \longrightarrow 0}$, w.p.1, as
$n\longrightarrow \infty. $ \hfill$\Box$
\end{proof}

We are now in a position to state the converges in probability on $S_n(R_{1},...,R_{n-1}).$

\begin{theorem}
Let $\{X_{n}\}$ be a sequence of
independent, identically distribution with \newline$E(\mid
X_{k}\mid)< \infty$, $E(X_{k})=\mu $ and $P(\textrm Max
R_i\longrightarrow 0)=1$. Then:
 $$S_n(R_{1},...,R_{n-1})\rightarrow \mu \quad
 \rm{ in  \ probability}.$$
\end{theorem}
\begin{proof}
The proof of theorem is an easy consequence of the
above lemmas. \hfill$\Box$
\end{proof}

\subsection{Equality}

One may guess that $S_n(R_{1},...,R_{n-1})$ have
minimum variance (as is the case for $\overline{x}$) but following theorem shows that this is not true.

\begin{theorem}
 $S_n(R_{1},...,R_{n-1})$ does not have minimum variance in class $\mathcal{C}$.
$$\mathcal{C}=\left\{S_n(W_{1},...,W_{n-1}):\mathbf{W_i}=V_{(i)}-V_{(i-1)},V_{n}=1, V_{(0)}=0\right\},$$
where $V_{i}$ have a  power distribution  with parameter $\theta$.
\end{theorem}

\begin{proof} We have
$V_{\theta}(S_n(W_{1},...,W_{n-1}))=(\sum_{i=1}^{n}EW_{i}^{2})\sigma^2$,
and variance for RWA in class $\mathcal{C}$ is
$$V_{\theta}(S_n(W_{1},...,W_{n-1}))=\sigma^2(\frac{2n\theta}{\theta+2}-\frac{2n\theta}{n\theta+1}+1$$
$$-\sum_{i=1}^{n-1}\sum_{k=0}^{n-i-1}\frac{2n!\theta^{2}(-1)^{n-i-1-k}}{(i\theta+1)(i-1)!k!(n-i-1-k)(n\theta-k\theta+2)}).$$

Differentiating $V_{\theta}(S_n(W_{1},...,W_{n}))$ with respect to
$\theta$ and evaluating at $\theta=1$, requires a thedious work,
which finally results in

\begin{eqnarray*}
 \frac{\partial
V_{\theta}(S_n(W_{1},...,W_{n-1}))}{\partial\theta}\mid _{\theta=1}
&=&\frac{2n-2(n+2)\sum_{i=2}^{n+1}\frac{1}{i}}{(n+1)(n+2)^{2}}.
\end{eqnarray*}

When $n \geq 2$ the last expression is negative, so in $\theta=1$,
$V_{\theta}(S_n(W_{1},...,W_{n-1}))$ is decreasing. On the other
hand, this function on $\theta$ is continues, so there is some
$\theta$ in which
the value of function is less than the value at $\theta=1$.\hfill$\Box$  \end{proof}

As can be seen from the below figure, the function becomes increasing
when the interval $(1,2)$ is omitted from the domain. Thus in that
case the minimum value of the function on the near domain is
achieved at $\theta=1$.

\begin{figure}[h]
\begin{center}
\scalebox{0.55}{\includegraphics{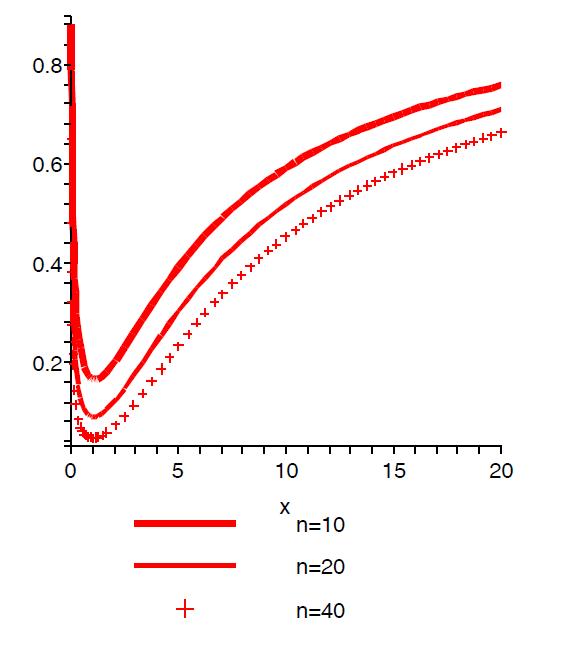}}
\end{center}
\caption{$V_{\theta}((S_n(W_{1},...,W_{n})))$\quad
$\theta\geq 1,\  n=10,\  n=20,\  n=40$.} \label{myfig1}
\end{figure}


\newpage

\begin{thebibliography}{}
\bibitem{} Devroye, L. (1981). Laws of the iterated logarithm for order
statistics of uniform spacings.  Annals of Probability 9, 860--867.


 \bibitem {}  Homei, H. (2012). Randomly weighted averages with beta random proportions. Statistics and Probability Letters 82, 1515--1520.

\bibitem{} Soltani, A.R. \&  Homei, H. (2009a). Weighted averages with random proportions
 that are jointly uniformly distributed over the unit simplex. Statistics and Probability Letters 9, 1215--1218.

\bibitem{} Soltani, A.R. \& Homei. H. (2009b). A generalization for two- sided power distributions and
adjusted metod of moments statistics. Statistics and Probability Letters 43, 611--620.

\bibitem{} Soltani, A. R. \& Roozegar, R. (2012). On distribution of randomly ordered uniform incremental
weighted average: Divided difference approach. Statistics and Probability Letters 82,  1012--1020.


\bibitem{} Pruitt, W.E. (1966). Summability of independent random variables. Journal of Mathematics and Mechanics 15, 769--776.


\bibitem{} Van Assche, W. (1987). A random variable uniformly
distributed between two independent random variables. Sankhy\={a}: The Indian Journal of Statistics,
Series  A  49, 207--211.



\bibitem{} Weisberg, H. (1971). The distribution of lineare combinations of order statistics from the
uniform distribution. Annals of  Mathematical  Statististics 42, 265--270.

\end{thebibliography}
\end{document}